\documentclass[12pt]{article}
\usepackage[english]{babel}

\usepackage[letterpaper,top=2cm,bottom=2cm,left=3cm,right=3cm,marginparwidth=1.75cm]{geometry}

\usepackage{amsmath}
\usepackage{amssymb}
\usepackage{graphicx}
\usepackage[colorlinks=true, allcolors=blue]{hyperref}
\usepackage{multicol}
\newtheorem{theorem}{Theorem}
\newtheorem{definition}{Definition}

\newtheorem{proposition}{Proposition}
\newtheorem{example}{Example}
\newtheorem{remark}{Remark}
\newenvironment{proof}[0]
{\medskip \noindent {\it Proof.} \ }{\ \hfill $\Box$\bigskip}
\title{The R-IFSs and their attractors}
\author{Nguyen Viet Hung, Mai The Duy and Vu Thi Hong Thanh}
\newcommand{\Addresses}{{% additional braces for segregating \footnotesize
  \bigskip
  \footnotesize

  Nguyen Viet Hung,\textsc{\textbf{specific address} }\par\nopagebreak
  \textit{E-mail address}: \texttt{nvh0@yahoo.com}

  \medskip

  Mai The Duy , \textsc{FPT Education, \textbf{specific address}}\par\nopagebreak
  \textit{E-mail address}: \texttt{duymt2@fe.edu.vn} or \texttt{vnwfpb@fgmail.com} \par\nopagebreak
  \textit{Telephone number}: (+84)902516790

  \medskip

  Vu Thi Hong Thanh, \textsc{Department of Mathematics, Vinh University, 182 Le Duan, Vinh city, Nghe An province}\par\nopagebreak
  \textit{E-mail address}: \texttt{thanhvth@vinhuni.edu.vn}

}}

\begin{document}
\maketitle

\begin{abstract}

This paper introduces a new class of iterated function systems (IFSs) called R-IFSs, which include both rotation/reflection maps and contraction maps. The study of R-IFSs is motivated by the recent research direction on enriching IFSs by adding other types of mappings. The paper investigates the existence and properties of the semi-attractor and compact invariant sets of R-IFSs, as well as the class of minimal invariant sets of R-IFSs. The paper also provides a familiar setting that 
\vspace{0.5cm}
is an invariant set of R-IFS but not an invariant set of any IFS. 
\vspace{0.45cm}
\textbf{Mathematics Subject Classification Primary:} 28A80; Secondary: 37B55 · 68R15

\end{abstract}

\section{Introduction} 

Fractals have wide applications in biology, computer graphics, quantum physics, and several other areas of applied
sciences \cite{barnsley2014fractals}. Especially, there are important consequences of fractals in several topics of applied sciences. Therefore, studying the existence or broadly the class of invariant sets or finding ways to build them has always been a problem of interest to many mathematicians in the world. Iterated function system (abbreviated as IFS) is a formalism for generating invariant sets in general and exactly self-similar fractals in particular based on the work of Hutchinson \cite{Hut81}, and popularized by Barnsley \cite{barnsley2014fractals}. 
To expand fractal sets, IFSs have been extended in different ways. 
For example, instead of considering a finite system of contraction mappings, one considers an infinite system or countable contraction mappings (\cite{barrozo2014countable}, \cite{mauldin1995infinite}). In the other direction,  instead of looking at mappings on metric spaces (\cite{kumari2019multi}), one can consider mappings on a class of expending spaces, such as $b$-metric, pseudo-metric, quasi-metric...or replace the Banach contraction by other contraction conditions such as $\varphi$ -contraction, $\psi$-contraction,... (\cite{priyanka2020iterated}) and get many results exciting. In particular, in recent years, a new research direction has appeared in the expansion of the concept of IFS. They enrich IFS by adding to it with other types of mapping that are not Banach contractions.  Specifically, K. Lesniak and
N. Snigireva in \cite{lesniak2022iterated} consider the case when IFS $\mathcal{F}$ is a finite family
of Banach contractions and $\mathcal{G}$ is a \textbf{singleton} consisting of periodic symmetry (that is, the isometry whose certain iteration is the identity) with a fixed point. Roughly speaking, the main results of \cite{lesniak2022iterated} show that in considered
cases, the semiattractor $A^*_{\mathcal{F}\cup \mathcal{G}}$ is in fact the attractor of certain IFSs which are somehow related to  $\mathcal{F}$ and $\mathcal{G}$. Moreover, K. Lesniak and N. Snigireva proved that the semiattractor $A^*_{\mathcal{F}\cup \mathcal{G}}$  can be recovered by the famous chaos game algorithm. After that, to motivate  the recent paper of  K. Lesniak and N. Snigireva, F. Strobin (in \cite{strobin2021contractive}) investigate the properties of the semiattractor $A^*_{\mathcal{F}\cup \mathcal{G}}$ of an IFS $\mathcal{F}$ enriched by some other IFS $\mathcal{G}$. They extend the framework. Instead of a single periodic isometry, in their main results,  they considers that $\mathcal{G}$ is just an IFS consisting of nonexpansive maps. He also weakens the contractive assumptions on maps
from the IFS $\mathcal{F}$ and allows $\mathcal{F}$ to be infinite. They use the presented machinery to prove that the so called lower transition attractors due to Vince are semiattractors of enriched IFSs. 

 Initiation, an iterated function system is defined as a finite set of contraction mappings on a  complete metric space $X,$ and iteration is defined as a sequential composition of these contraction mappings, where each mapping has some nonzero probability of being used at every iteration. Based on this notion and the famous Banach fixed point theorem \cite{Banach1922}, Hutchinson showed that for the complete metric space $X,$ such a system of functions $\{f_1,...,f_n\}$ has a unique closed bounded set $F,$ satisfying $$F=\bigcup_{i=1}^n f_i(F).$$ The set $F$ is called an invariant set or an attractor set (or sometimes,   a fractal set) of the iterated function system. The iterated function systems are introduced as a unified way of generating a broad class of fractals.  

  To enrich the initial IFSs, on the direction of extending the initiate concept of IFS by adding other types of mapping, as in \cite{lesniak2022iterated} and \cite{strobin2021contractive}, in this paper, we introduce the class of iterated function systems 
 including both rotation/reflection maps and contraction maps, denoted by R-IFSs. Thus, instead of adding  a single isometric mapping to IFS as in \cite{lesniak2022iterated}, we add to IFS by a finite family of isometric mappings. We study the problem:  What can we say about the semiattractor and the invariant sets of  IFSs. We also provide a familiar setting that is an invariant set of R-IFS but is not an invariant set of any IFS.
 
The remaining structure of this paper is as follows. First,  the notion of R-IFS and examples are introduced. Then, we prove the existence and properties of the semiattractor  and the invariant sets of R-IFSs. In this aspect, we prove that the class of invariant sets of IFSs is the subset of the class of  minimal invariant sets of the R-IFSs. The  minimal invariant set of new IFS is a semiattractor. Finally, we show that there are minimal invariant sets of the R-IFSs, which are not the invariant sets of any IFSs. We also illustrate our main results with several
examples that are related to classical fractals such as the Sierpinski gasket with two mappings and the
Cantor target. 

--------------------------

\section{The main theorems}

Throughout this paper, the standard letters $\mathbb{R}, \mathbb{R}^d, \mathbb{N}$  and $ \mathbb{N}^\ast$ will denote the set of all real numbers, the set of all the $d$-tuples of real numbers ($d$ is a positive integer number), the set of all nonnegative integer numbers and the set of all positive integer numbers, respectively.

Given a metric space $X,$ by $2^X, \mathcal{K}(X)$ and $C_b(X)$ we will denote the hyperspaces of all subsets of $X,$ all nonempty and compact subsets of $X,$ and nonempty closed and bounded subsets of $X,$ respectively. The spaces $\mathcal{K}(X), C_b(X)$ will be considered as metric spaces endowed with the Hausdorff–Pompeiu metric. It is well known, that $\mathcal{K}(X)$ and $C_b(X)$ are complete provided that $X$ is complete. For $A \subset X,$
by $cl A$ we will denote the closure of the set $A.$

\begin{definition}
Let $(X, d)$ be a  metric space and a mapping  $f: X \rightarrow X$.

i) The mapping  $f$ is called a Banach contraction on $X$ if   there is a number $ \alpha \in  [0, 1)$   such that  for any $x, y \in X$ the following holds:
$$d(f(x), f(y)) \leq \alpha d (x, y). $$  The number $\alpha $ is called a contraction ratio of $f$.

ii) The mapping $f$  preserves the distances between points is called an isometric mapping. This means if $x, y \in X$ then
$$d(f(x), f(y)) = d (x, y). $$
\end{definition}

\begin{definition}
By an iterated function system  (IFS for short) we will mean
any finite family $\mathcal{F}=\{f_1,...,f_n\}$ of contraction selfmaps of a metric space $X.$
If $\mathcal{F}$ is an IFS, then

i) by $F_{\mathcal{F}}$ we  denote the Hutchinson operator to $\mathcal{F},$ that is, the map
$F_{\mathcal{F}}: \mathcal{K}(X) \rightarrow \mathcal{K}(X)$  defined by: $$\forall A\subset X: F_{\mathcal{F}}(A) =\bigcup_{f_i\in \mathcal{F}}f_i(A).$$ 

ii) by $\widetilde {F}_{\mathcal{F}}$ we denote the weak Hutchinson operator to $\mathcal{F},$ that is, the map
$\widetilde {F}_{\mathcal{F}}: 2^X \rightarrow 2^X$  defined by: $$\forall A\subset X: \widetilde {F}_{\mathcal{F}}(A) =cl(\bigcup_{f_i\in \mathcal{F}}f_i(A)).$$ 
\end{definition}
We also write $F^n$ for the $n$-fold composition of $F.$

We recall  a definition of Kuratowski limits (see, e.g., \cite{beer1993topologies},  \cite{lasota2000attractors} for detailed discussion).
\begin{definition} Let $(S_n)$ be a sequence of subsets of a metric space $X.$ The upper Kuratowski limit of $(S_n)$ is defined by
$$Ls(S_n) := \{y \in X : y = \lim_{n \rightarrow \infty} x_n \text{ for some sequence } (x_n) $$
$$\text{ so that } x_n \in S_n \text{ for infinitely many } n \in  \mathbb{N}\}$$
and the lower Kuratowski limit of $(S_n)$ is defined by
$$Li(S_n) := \{y \in  X : y = \lim_{n \rightarrow \infty} x_n \text{ for }  x_n \in  S_n \text{ for all } n \in  \mathbb{N}\}.$$
If $Ls(S_n) = Li(S_n),$ then its common value, denoted by $Lt(S_n),$ is called the Kuratowski limit of $(S_n).$
\end{definition}

Now we recall the notion of semiattractor of an IFS. We refer the reader to \cite{lasota2000attractors} or \cite{lesniak2022iterated} for details.

\begin{definition}
Let IFS  $\mathcal{F}=\{f_1,...,f_n\}$ be  a system of continuous maps and $F_{\mathcal{F}}$ the associated Hutchinson operator. A nonempty closed set $A^* \subset  X$ is

i) an \textbf{invariant set } of $\mathcal{F}$ if $F_{\mathcal{F}}(A^*) = A^*;$

ii) a Hutchinson \textbf{attractor } of $\mathcal{F}$ if $F_{\mathcal{F}}^n(S) \rightarrow A^*$ for all nonempty closed and bounded $S\subset X$
where the convergence takes place in the Hausdorff metric;

iii) a \textbf{semiattractor } if $\bigcap_{x\in X} Li(F^n_{\mathcal{F}}(\{x\})) = A^*;$

iv)  a \textbf{minimal invariant set }  of $\mathcal{F}$ if  $A^* \subset S$ for any invariant set $S$ of $\mathcal{F}.$

\end{definition}
The relation between invariant sets, attractors, and semiattractors is summarised below.

\begin{theorem} \label{dl1} (\cite{myjak2003attractors})
Let $\mathcal{F}=\{f_1,...,f_n\}$ be an IFS and $A^* \subset  X$ a nonempty closed set.

i) A Hutchinson attractor of $\mathcal{F}$ is a unique nonempty closed bounded invariant set.

ii) A semiattractor $A^*$ of $\mathcal{F}$ is the smallest invariant set, i.e., $A^* \subset C$ for every
nonempty closed invariant set $C.$ 

iii) A Hutchinson attractor of  $\mathcal{F}$  is a compact semiattractor.

iv) If $A^*$ is a semiattractor of $\mathcal{F},$ then $Lt F^n(S) = A^*$ for every nonempty $S \subset  A^*.$
  
\end{theorem}

Denote $C$ as  a group of isometric  mappings $T$ on $\mathbb{R}^d$  such that $T(0) = 0.$ 

\begin{definition}
An R-IFS  is a function system consisting of $m$ isometric mappings $R_1,...,R_m$ belonging to $C $ and $n$ contraction mappings $f_1,..., f_n$ ($m, n>0$) on $\mathbb{R}^d $ and denoted  this R-IFS  by  $\mathcal{F_R}=\{ R_1,...,R_m, f_1,... f_n \}.$     
\end{definition}

For an IFS of finite contraction maps on a complete metric space, there is a unique invariant set of this IFS. But, for each R-IFS, we get the following result.

\begin{theorem}\label{dl2}
There exists an infinite number of compact invariant sets of the R-IFS  $\mathcal{F_R}=\{ R_1,...., R_m, f_1,..., f_n \}$.    
\end{theorem} 

\begin{proof}
Let $r_j$ be the contraction ratio of $f_j$, for $j=1,...,n$. For any $x\in \mathbb {R}^d$, we have    $$||f_j(x)|| \leq ||f_i(0)||+||f_j(x)-f_j(0)||\leq ||f_j(0)||+ r_j||x-0||= ||f_j(0)||+r_j||x||.$$  Set $\lambda_0=\max \{\displaystyle \frac{||f_j(0)||}{1-r}: j=1,2,...,n\},$ where 
$r=\max\{r_1,...,r_n\}.$ Then for each $j\in \{1,...,n\},$ 
$f_j(S(\lambda)) \subseteq S(\lambda)$, where $S(\lambda)$ is a closed ball with center $0$ and radius $\lambda >\lambda_0$. At the same time, for each $i\in \{1,2,...,m\},$  we have $R_i(S(\lambda))=S(\lambda)$. Therefore, for each $\lambda>\lambda_0$ the closed ball with center $0,$ radius  $\lambda$ is an  invariant set of the R-IFS $\{ R_1,...,R_m, f_1,...,f_n \}. $  Thus, there are an infinite amount of compact invariant sets of this $\mathcal{F_R}.$
\end{proof}

For $k\in \mathbb{N}$, let  $R_u=R_{u_0}\circ R_{u_1}\circ ... \circ R_{u_k}$ where $u=u_0 ... u_k$ and $u_i=u_i^0 u_i^1...u_i^l, \ u_i^j\in \{1,...,m\}$. 

\begin{theorem}\label{dl3}
There exists the minimal invariant set $X_0$  of  R-IFS $\{ R_1,...,R_m, f_1,..., f_n \}.$  
\end{theorem} 

\begin{proof}
At first, we prove that  the limit $\lim\limits_{k\rightarrow\infty}  R_{u_0}f_{j_0}R_{u_1}f_{j_1}...R_{u_k}f_{j_k}(0)$ exists for  any finite  index $u_0,  ..., u_k$ with $u_i=u_i^0 u_i^1...u_i^l, \ u_i^j\in  \{1,...,m\}$ and $j_0,...,j_k \in \{1,...,n\}.$

Indeed, we have  $$||R_{u_0}f_{j_0}R_{u_1}f_{j_1}...R_{u_k}f_{j_k}(x)-R_{u_0}f_{j_0}R_{u_1}f_{j_1}...R_{u_k}f_{j_k}(0)||\leq {r^k}\lambda_0$$ for any $x\in \mathbb{R}^d,$ where $\lambda_0$ is a const positive number defined in the proof of Theorem \ref{dl2}, and $r=\max\{r_i: i\in \{j_0, j_1, ..., j_k\}\}$. So for any $p\in \mathbb{N},$ put  $x=R_{u_{k+1}}f_{j_{k+1}}...R_{u_{k+p}}f_{j_{k+p}}(0)$ then 
\begin{align*}
			&||R_{u_0}f_{j_0}R_{u_1}f_{j_1}...R_{u_{k+p}}f_{j_{k+p}}(0)-R_{u_0}f_{j_0}R_{u_1}f_{j_1}...R_{u_k}f_{j_k}(0)|=\\
			&=||R_{u_0}f_{j_0}R_{u_1}f_{j_1}...R_{u_k}f_{j_k}(x)-R_{u_0}f_{j_0}R_{u_1}f_{j_1}...R_{u_k}f_{j_k}(0)|| \leq r^k\lambda_0.
		\end{align*}
 By  Cauchy's general principle of convergence,  the limit   $\lim\limits_{k\rightarrow\infty}  R_{u_0}f_{j_0}R_{u_1}f_{j_1}...R_{u_k}f_{j_k}(0)$ exists.

Now let $X_0$   be a closure of  set  $$X^*=\{\lim\limits_{k\rightarrow \infty}  R_{u_0}f_{j_0}R_{u_1}f_{j_1}...R_{u_k}f_{j_k}(0): u_k=i^0_k ... i^{n_k}_k\in \bigcup_{t=0}^{\infty} \{1,...,m\}^t \}.$$

Then     $R_i(X^*)\subseteq X^*$ for all $i\in \{1,...,m\},$ \ $ f_j(X^*)\subseteq X^*$ for all $j\in \{1,...,n\}$ and $X^* \subset [\bigcup\limits_{i=1}^m R_i(X^*)]\bigcup [\bigcup\limits_{j=1}^n f_j(X^*)]$. So one has $$X_0=[\bigcup_{i=1}^m R_i(X_0)] \bigcup [\bigcup_{j=1}^n f_j(X_0)]. $$
This means $X_0$ is an invariant set  of the  R-IFS  $\{ R_1,...,R_m,\ f_1,...,f_n \}.$

 Now, we are going to show that $X_0 \subseteq X$  for any closed bounded set  
 $$X=[\bigcup_{i=1}^n R_i(X)]\  \ \bigcup \ \ [\bigcup_{j=1}^m f_j(X)].$$ Indeed, let $y=\lim\limits_{k\rightarrow\infty}  R_{u_0}f_{j_0}R_{u_1}f_{j_1}...R_{u_k}f_{j_k}(0)\in X_0$ then        $y \in X$ as  $$\lim_{k\rightarrow\infty}  R_{u_0}f_{j_0}R_{u_1}f_{j_1}...R_{u_k}f_{j_k}(0)=\lim_{k\rightarrow\infty}  R_{u_0}f_{j_0}R_{u_1}f_{j_1}...R_{u_k}f_{j_k}(x)$$ for any $x\in X$.

\end{proof}

\begin{remark} By  Theorem \ref{dl1}  and Theorem \ref{dl3}, the R-IFS $\mathcal{F_R}=\{ R_1,...., R_m, f_1,..., f_n \}$ has a compact  semiattractor.
\end{remark} 

The above proof also provides a constructive approach for the minimal invariant set. This approach is used to prove the following results as well as to draw the example of R-IFS.

The following result shows a necessary condition in which each  minimal invariant set of any R-IFS is the attractor of some IFS. Now, let     $$G=<R_1,..., R_n>=\{R_u:\  u \in \bigcup\limits_{t=0}^{\infty} \{1,...,m\}^t\}.$$

\begin{theorem}
If $G$  is a finite group of isometric mappings, then the minimal invariant set  $X_0$  of the R-IFS $\{R_1,..., R_m,\ f_1,..., f_m\}$  is an attractor of some iterated function system. 
\end{theorem} 
\begin{proof}
Assume that the finite group $G$ consisting of elements
 $\{g_1,...,g_l\}.$ Clearly that the identity mapping $id\in G.$ Consider  a collection of $l\times n$ mappings are defined by $\{g_i\circ f_j:\  (i,j) \in \{1,...,l\}\times\{1,...,n\}\},$ then this function system is an IFS.   Let $X_1$ be an invariant set of this iterated function system.  By the result of John E. Hutchinson  \cite{Hut81}, $X_1$  is a closure of the set  
   $$X^*=\{\lim\limits_{k\rightarrow \infty} g_{i_0}f_{j_0}g_{i_1}f_{j_1}...g_{i_k}f_{j_k}(0): i_0,i_1,...,i_k \in \{1,...,l\}, j_0,j_1,...,j_k \in \{1,...,n\} \}$$
   $$=\{\lim\limits_{k\rightarrow \infty} R_{u_0}f_{j_0}R_{u_1}f_{j_1}...R_{u_k}f_{j_k}(0), u_k=i^0_k ... i^{n_k}_k\in \bigcup_{t=0}^{\infty} \{1,...,m\}^t,j_0,j_1,...,j_n \in \{1,...,n\} \}$$
   is defined in Theorem 2. So  $X_1=X_0,$ where $X_0$ is the minimal invariant set of the R-IFS $\{ R_1,...,R_m, f_1,...,f_n \}.$
\end{proof}

\begin{remark}
As the consequence, the invariant set  respect to  IFS $\{f_1,..., f_n\}$ is the minimal invariant set of the R-IFS $\{id, f_1,..., f_n\}$, where $id$ is the identity mapping. Thus, we can consider that the class of R-IFS is broader than the class of IFS.
\end{remark}

\begin{example}
 In $\mathbb R^2$, the minimal invariant set  of R-IFS  consisting of two mappings  
 $$ \{ f_1(z)=e^{i2\pi/3}z , f_2(z)=\displaystyle \frac{z}{2} +1\} $$ 
is the invariant set of IFS $$\{f_2(z),f_1\circ f_2(z), f_1(z)\circ f_1(z)\circ f_2(z)\},$$ i.e., the invariant set, in this case, is the Sierpinski gasket, see Figure \ref{SierpinskiGasket}.
\end{example}

\begin{figure}
\begin{center}
%\href{https://colab.research.google.com/drive/1trPAcVTfr-t6MDydhKBkbHth2GQ0vgYH?usp=sharing}{\includegraphics[width=7.2cm,height=7.4cm]{Pic/Sirpinski.png}}
\includegraphics[width=7.2cm,height=7.4cm]{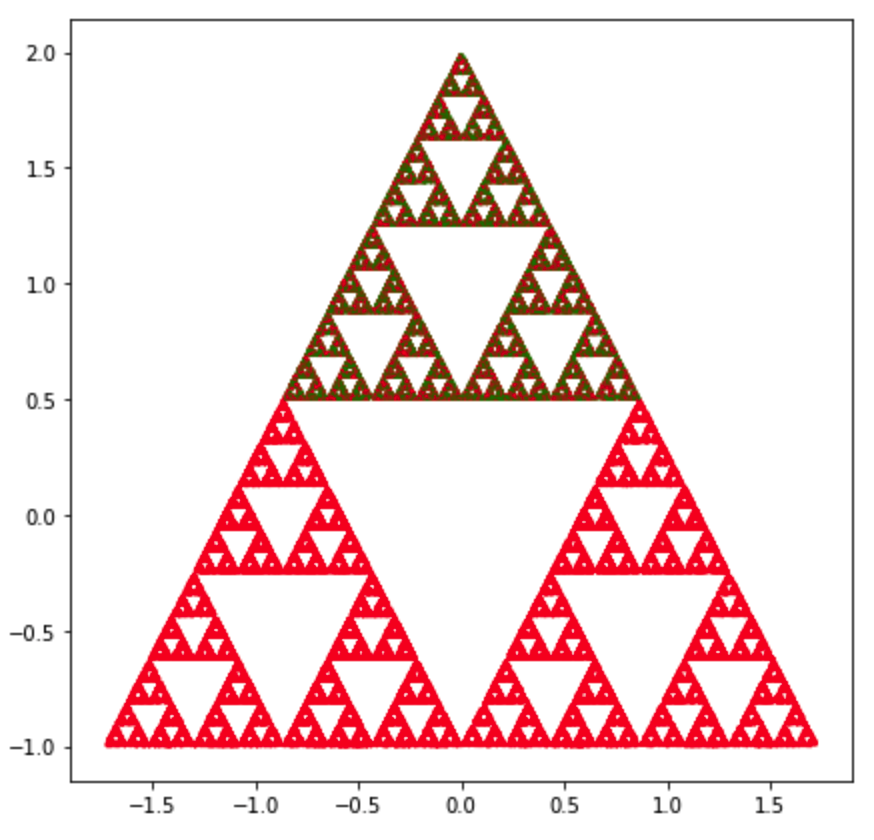}
\caption{The Sierpinski gasket with two mappings }
\label{SierpinskiGasket}
\end{center} 
\end{figure} 

\begin{example}
 A fractal $n$-gon is the invariant set of IFS  $ \{f_k(z) = \lambda z + c_k,k = 1, ...,n\}$ of which there is a rotation in the plane which acts transitively on the pieces, permuting them in an $n$-cycle. Bandt and Hung \cite{ngon} showed that any fractal $n$-gon is similar to a self-similar set 
 $$A =\bigcup_{k=1}^n f_k(A)$$
 where $f_k(z) = \lambda z + b^k$ for $k = 1,..., n$, where $b = cos(\frac{2\pi}{n})+i*sin(\frac{2\pi}{n})$ and $|\lambda| < 1.$ Now we are going to show that $A$ is the minimal invariant set of the R-IFS $$\{g(z)=b*z, f_1(z)=\lambda z + 1\}.$$ Indeed, as in the proof of Theorem 3, the minimal invariant set of the R-IFS  $$\{g(z)=b*z, f_1(z)=\lambda z + 1\}$$  is the invariant set of the IFS $\{f_k(z) = b^k\lambda z + b^k: \  k=1,..,n\}$. The invariant set  of the last IFS is the same as $A.$
\end{example}

 To show that the class of all invariant sets of IFSs is a proper subset of the class of all  minimal invariant sets of R-IFSs, at first, we recall the ‘Cantor target’, see the Figure \ref{Cantors}, which is an example satisfies the condition $$\dim_H(E \times F) = \dim_HE + \dim_HF.$$
 
The ‘Cantor target’ is the set swept out by rotating the middle third Cantor set about an endpoint. Thus, this set is the plane set given in polar coordinates by $$C = \{(r, \theta): r=||x||\ x \in F, 0 \leq \theta \leq 2\pi\}$$ where $F$ is the middle third Cantor set.

\begin{figure}
\begin{center}
%\href{https://colab.research.google.com/drive/1IPiukVANV93DzI14Kwc97qh7UefW2Jj6?usp=sharing}{\includegraphics[width=7.4cm,height=7.4cm]{Pic/Cantor_target2.png}}
\includegraphics[width=7.4cm,height=7.4cm]{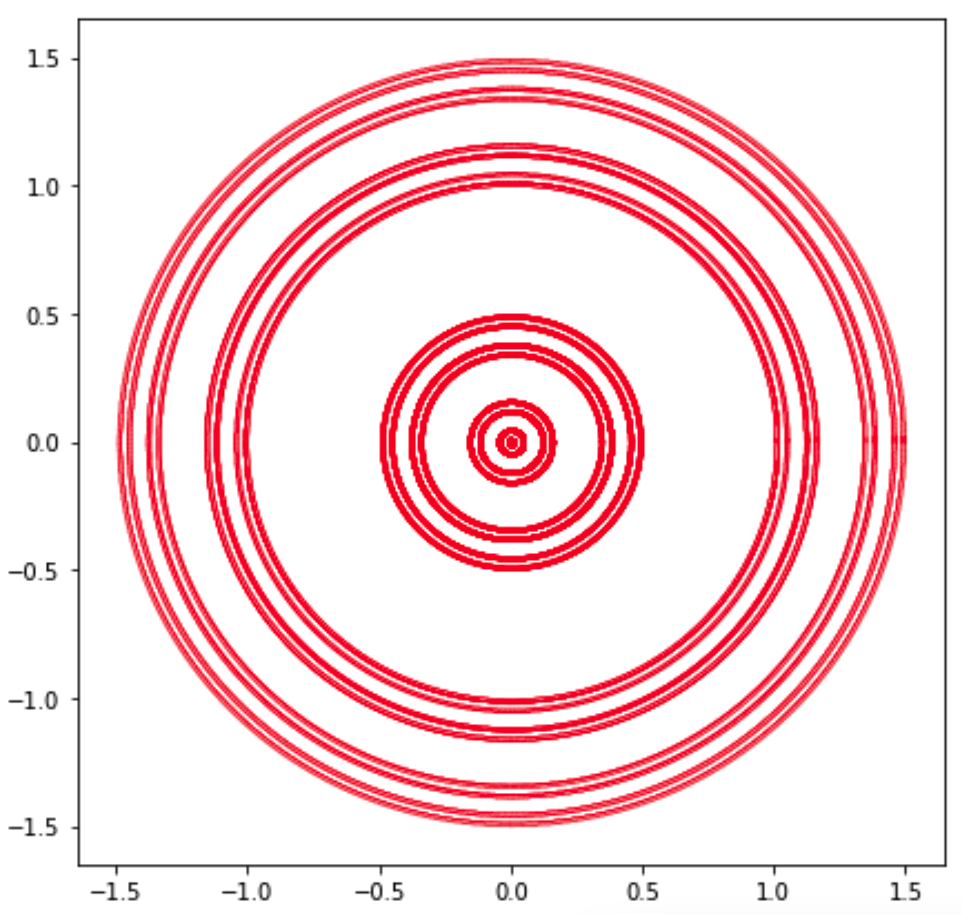}
\caption{The Cantor target p.104 \cite{falconer2004fractal}}
\label{Cantors}
\end{center} 
\end{figure}  

\begin{figure}
\begin{center}
\includegraphics[width=8.0cm,height=7.4cm]{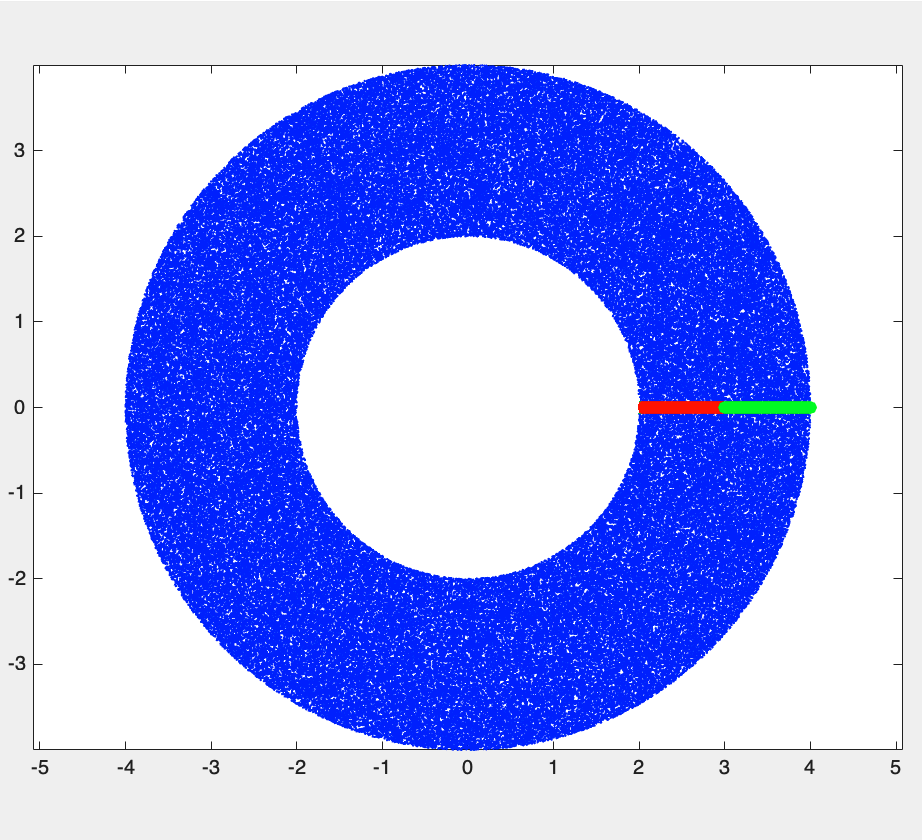}
\caption{The "ring" is the minimal invariant of the R-IFS $\left \{ g;\ f_1(x,y) =\big(\frac{\sqrt{x^2+y^2}}{3}, 0 \big)+\big(1, 0\big);\  f_2(x,y)=\big(\frac{\sqrt{x^2+y^2}}{3}, 0 \big)+\big(2, 0\big)\right \} $  }
\label{ring}
\end{center} 
\end{figure}  

\begin{figure}
\begin{center}
\includegraphics[width=8.0cm,height=7.4cm]{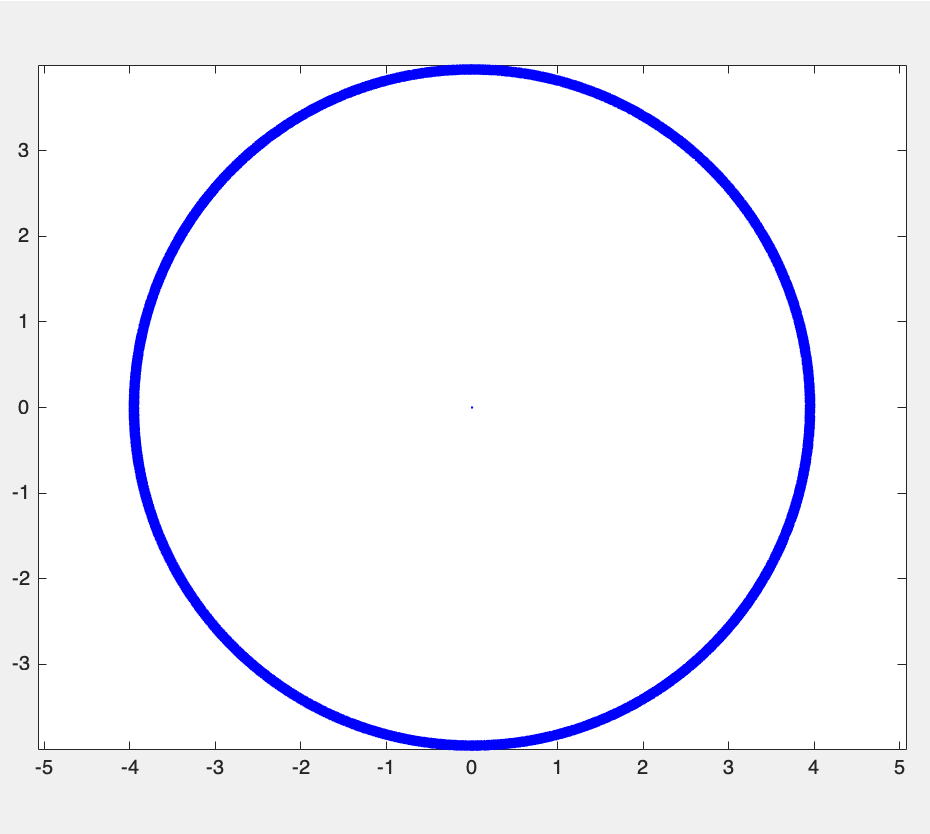}
\caption{The circle is the minimal invariant of the R-IFS
  $ \left \{ g;\ f_1(x,y) =\big(\frac{\sqrt{x^2+y^2}}{3}, 0 \big)+\big(2/3, 0\big)\right \}$   }
\label{Cir}
\end{center} 
\end{figure} 
    Now, we are going to show that the Cantor target is an invariant set induced by R-IFS but is not the invariant set of any IFS. 
\begin{proposition}
There are R-IFSs of which the minimal invariant sets are not the invariants of any IFS consisting of a system of Banach contractions.

\end{proposition}
\begin{proof}
In  $\mathbb{R}^2,$ let  $g$ be a rotation of  $\alpha. 2\pi,$  where  $\alpha$  is an irrational number. Consider the R-IFS
 $\left \{ g;\ f_1(x,y) =\big(\frac{\sqrt{x^2+y^2}}{3}, 0 \big);\  f_2(x,y)=\big(\frac{\sqrt{x^2+y^2}}{3}, 0 \big)+\big(\frac{2}{3}, 0\big)\right \}.$  Then the minimal invariant $X_0$ of this R-IFS  is a union of uncountable discrete circles $C_{x_c}$ with center  $O(0,0)$  radius  $x_c$  where  $x_c$  in the Cantor set  $X,$  i.e.,  $X_0=\displaystyle  \bigcup_{ x_c\in X}C_{x_c}.$  By Kronecker’s density theorem \cite{hawkins2013mathematics}, the closures of  $\{g^n(C_{x_c})\}$  is a circle  $C_{x_c}$ for $n=1,2,...$ Since  $X_0$  is close, we get $C_{x_c}\subset X_0.$

Assuming that there is an IFS  $\{f_1,...,f_n\}$  such that  $X_0=\bigcup\limits_{i=1}^n f_i(X_0).$  Now, take  $||x_0||=\max arg\{||x||: \ x\in X_0\}.$  Then
$C_{x_c}\subset X_0\subset D_{x_0},$ where  $D_{x_0}$  is the disc of center  $O(0,0)$  and radius  $||x_0||.$ Hence,   $f_i(C_{x_0})\subset f_i(X_0)\subset f_i(D_{x_0})$   for all  $i=1,..., n.$

Note that for any $i=1,..., n$  the boundary of  $f_i(D_{x_0})$  is the subset of image of the boundary of  $D_{x_0}$  by  $f_i.$  \ (1)

For any  $i=1,..., n,$  since   $f_i(C_{x_0})$  is connected,  $f_i(C_{x_0})$  could only belong to one circle  $C_{x^*_i}$  for an fixed  $x^*_i$  in the Cantor set. By (1),  $f_i(D_{x_0})\subset C_{x^*_i}.$ Hence  $f_i(X_0)\subset C_{x^*_i}.$

Therefore,  $X_0=\bigcup\limits_{i=1}^n f_i(X_0)\subset \bigcup\limits_{i=1}^n C_{x^*_i}.$ This contradict to  $X_0=\displaystyle \bigcup_{ x_c\in X}C_{x_c}.$ 
\end{proof}

\begin{example}\label{newex} 
Similar to the above proof, given $g$ a rotation of $\alpha.2\pi$ with an irrational number $\alpha$,  the minimal invariant of the R-IFS $$\left \{ g;\  f_1(x,y) =\big(\frac{\sqrt{x^2+y^2}}{3}, 0 \big)+\big(1, 0\big); \  f_2(x,y)=\big(\frac{\sqrt{x^2+y^2}}{3}, 0 \big)+\big(2, 0\big)\right \}$$  is a ring (Figure \ref{ring}), and the minimal invariant of the R-IFS
 $$\left \{ g;\  f_1(x,y) =\big(\frac{\sqrt{x^2+y^2}}{3}, 0 \big)+\big(2/3, 0\big)\right \}$$  is a circle (Figure \ref{Cir}).
 \end{example}

\section{Conclusion}

This paper introduces a new class of iterative function systems, R-IFS, generated by combining rotation/reflection mappings with contraction mappings. Within this class, there exist infinitely many invariant sets, but each R-IFS has a unique minimal invariant set. Any invariant set of a given IFS can also be the minimal invariant set of a corresponding R-IFS. Notably, some symmetric self-similar sets can be invariant sets of R-IFS with fewer mappings than those required by the original IFS. Furthermore, this paper presents examples that belong to the new class but cannot be the invariant sets of any IFS.

%Open questions:
%- Could we extend the minimal invariant set to other general IFS?

%- How can we calculate the Haufdorff dimension of the minimal invariant set when they are not the attractor of IFS?

\Addresses

\begin{thebibliography}{KCCH19}

\bibitem[Ban22]{Banach1922}
S.~Banach.
\newblock Sur les opérations dans les ensembles abstraits et leur application
  aux équations intégrales.
\newblock {\em Fundamenta Mathematicae}, 3(3):133--181, 1922.

\bibitem[Bar14]{barnsley2014fractals}
Michael~F Barnsley.
\newblock {\em Fractals everywhere}.
\newblock Academic press, 2014.

\bibitem[Bee93]{beer1993topologies}
Gerald Beer.
\newblock {\em Topologies on closed and closed convex sets}, volume 268.
\newblock Springer Science \& Business Media, 1993.

\bibitem[BH08]{ngon}
Christoph Bandt and Nguyen~Viet Hung.
\newblock Fractal $n$-gons and their mandelbrot sets.
\newblock {\em Nonlinearity}, 21(11), 2008.

\bibitem[BM14]{barrozo2014countable}
Mar{\'\i}a~Fernanda Barrozo and Ursula Molter.
\newblock Countable contraction mappings in metric spaces: invariant sets and
  measure.
\newblock {\em Central European Journal of Mathematics}, 12:593--602, 2014.

\bibitem[Fal04]{falconer2004fractal}
Kenneth Falconer.
\newblock {\em Fractal geometry: Mathematical foundations and applications}.
\newblock John Wiley \& Sons, 2004.

\bibitem[Haw13]{hawkins2013mathematics}
Thomas Hawkins.
\newblock The mathematics of frobenius in context.
\newblock {\em Sources and Studies in the History of Mathematics and Physical
  Sciences. Springer}, 2013.

\bibitem[Hut81]{Hut81}
John~E. Hutchinson.
\newblock Fractals and self similarity.
\newblock {\em Indiana University Mathematics Journal}, 30(3):713--747, 1981.

\bibitem[KCCH19]{kumari2019multi}
Sudesh Kumari, Renu Chugh, Jinde Cao, and Chuangxia Huang.
\newblock Multi fractals of generalized multivalued iterated function systems
  in b-metric spaces with applications.
\newblock {\em Mathematics}, 7(10):967, 2019.

\bibitem[LM00]{lasota2000attractors}
Andrzej Lasota and Jozef Myjak.
\newblock Attractors of multifunetions.
\newblock {\em Bulletin of the Polish Academy of Sciences. Mathematics},
  48(3):319--334, 2000.

\bibitem[LS22]{lesniak2022iterated}
Krzysztof Le{\'s}niak and Nina Snigireva.
\newblock Iterated function systems enriched with symmetry.
\newblock {\em Constructive Approximation}, 56(3):555--575, 2022.

\bibitem[Mau95]{mauldin1995infinite}
R~Daniel Mauldin.
\newblock Infinite iterated function systems: theory and applications.
\newblock In {\em Fractal geometry and stochastics}, pages 91--110. Springer,
  1995.

\bibitem[MS03]{myjak2003attractors}
J{\'o}zef Myjak and Tomasz Szarek.
\newblock Attractors of iterated function systems and markov operators.
\newblock In {\em Abstract and Applied Analysis}, volume 2003, pages 479--502.
  Hindawi, 2003.

\bibitem[PS20]{priyanka2020iterated}
SC~Priyanka and R~Shrivastava.
\newblock Iterated function system consisting of kannan contraction in
  complete-metric space.
\newblock {\em Turkish Journal of Computer and Mathematics Education
  (TURCOMAT)}, 11(3):2204--2214, 2020.

\bibitem[Str21]{strobin2021contractive}
Filip Strobin.
\newblock Contractive iterated function systems enriched with nonexpansive
  maps.
\newblock {\em Results in Mathematics}, 76(3):153, 2021.

\end{thebibliography}
\end{document}